\documentclass[preprint,nonatbib]{elsarticle}
\usepackage{fullpage}
\usepackage{standalone}
\usepackage{import}

\makeatletter
\let\c@author\relax
\makeatother

\usepackage[table]{xcolor}
\usepackage{amsmath}
\usepackage{amsthm}
\usepackage{amssymb} 
\usepackage{hyperref}
\usepackage{graphicx}
\usepackage{float}
\usepackage{tikz-cd}
\usepackage{caption}
\usepackage{subcaption}
\usepackage{array}
\usepackage{extarrows}
\usepackage{multirow}
\usepackage{tikz}
\usepackage{makecell}
\usepackage{hhline}
\usepackage{ mathdots }
\usepackage{cleveref}
\usepackage{enumitem}
\usetikzlibrary{cd}
\usetikzlibrary{graphs}
\usetikzlibrary {graphs.standard}
\usetikzlibrary{shapes.geometric}

\DeclareMathOperator{\LS}{RP}

\DeclareMathOperator{\ROS}{ROS}
\DeclareMathOperator{\SROS}{SROS}

\allowdisplaybreaks[3]

\newtheorem{theorem}{Theorem}[section]
\newtheorem{lemma}[theorem]{Lemma}

\newtheorem{corollary}[theorem]{Corollary}
\newtheorem{conjecture}[theorem]{Conjecture}

\theoremstyle{definition}
\newtheorem{definition}[theorem]{Definition}
\newtheorem{example}[theorem]{Example}

\newtheorem{observation}[theorem]{Observation}

\Crefname{conjecture}{Conjecture}{Conjectures}

\numberwithin{equation}{section}

\usepackage{biblatex} 
\hypersetup{pdfauthor=author}

\makeatletter
\def\ps@pprintTitle{%
  \let\@oddhead\@empty
  \let\@evenhead\@empty
  \def\@oddfoot{\reset@font\hfil\thepage\hfil}
  \let\@evenfoot\@oddfoot
}
\makeatother

\begin{document}

\title{Further results on latin squares with disjoint subsquares using rational outline squares}

\author[1]{Tara Kemp\corref{cor1}}
\ead{t.kemp@uq.net.au}
\author[1]{James Lefevre}
\ead{j.lefevre@uq.edu.au}

\affiliation[1]{organization={School of Mathematics and Physics, ARC Centre of Excellence, Plant Success in Nature and Agriculture},
    addressline={The University of Queensland},
    city={Brisbane},
    postcode={4072},
    country={Australia}}

\cortext[cor1]{Corresponding author}

\begin{abstract}
    In this paper we consider the problem of finding latin squares with sets of pairwise disjoint subsquares. We develop a new necessary condition on the sizes of the subsquares which incorporates and extends the known conditions. We provide a construction for the case where all but two of the subsquares are the same size, and in this case the condition is sufficient. We obtain these results using symmetric rational outline squares, and additionally provide several new results and extensions to this theory.
\end{abstract}

\maketitle

\section{Preliminaries}

A \emph{latin square} of order $n$ is an $n\times n$ array with each of $n$ symbols occurring exactly once in each row and column. Throughout, we assume that the symbols are from the set $[n] = \{1,2,\dots,n\}$, and we index the rows and columns using the same set. An $m\times m$ sub-array of a latin square which is itself a latin square is an order $m$ \emph{subsquare}. There has been much work on finding latin squares without subsquares of predetermined sizes (see \cite{allsop2023latin}) and the probability of a random latin square having any proper subsquares (see \cite{gill2024canonical}). In this paper, we consider the existence of latin squares with sets of pairwise disjoint subsquares. Subsquares are considered disjoint when they share no rows, columns or symbols.

\begin{example}
    The latin squares in \Cref{fig: nonconsecutive square example,fig: square example} are of order 8 and 9 respectively with disjoint subsquares of orders 2 and 3.

    \begin{figure}[h]
    \centering
    \begin{subfigure}{0.3\textwidth}
        \centering
        $\arraycolsep=4pt\begin{array}{|c|c|c|c|c|c|c|c|} \hline
        \cellcolor{lightgray}1 & \cellcolor{lightgray}6 & 2 & 3 & 4 & 5 & 8 & \cellcolor{lightgray}7 \\ \hline
        4 & 5 & 1 & 6 & 7 & 2 & 3 & 8 \\ \hline
        \cellcolor{lightgray}6 & \cellcolor{lightgray}7 & 3 & 8 & 2 & 4 & 5 & \cellcolor{lightgray}1 \\ \hline
        3 & 4 & 6 & 7 & 8 & 1 & 2 & 5 \\ \hline
        \cellcolor{lightgray}7 & \cellcolor{lightgray}1 & 8 & 2 & 5 & 3 & 4 & \cellcolor{lightgray}6 \\ \hline
        5 & 2 & 7 & 1 & 3 & 8 & 6 & 4 \\ \hline
        2 & 8 & \cellcolor{lightgray!40}4 & \cellcolor{lightgray!40}5 & 1 & 6 & 7 & 3 \\ \hline
        8 & 3 & \cellcolor{lightgray!40}5 & \cellcolor{lightgray!40}4 & 6 & 7 & 1 & 2 \\ \hline
        \end{array}$
        \caption{An order 8 latin square with two disjoint subsquares}
        \label{fig: nonconsecutive square example}
    \end{subfigure}
    \quad\quad\quad
    \begin{subfigure}{0.3\textwidth}
        \centering
    $\arraycolsep=4pt\begin{array}{|c|c|c|c|c|c|c|c|c|} \hline
    \cellcolor{lightgray}1 & \cellcolor{lightgray}2 & \cellcolor{lightgray}3 & 6 & 9 & 8 & 5 & 7 & 4 \\ \hline
    \cellcolor{lightgray}3 & \cellcolor{lightgray}1 & \cellcolor{lightgray}2 & 7 & 8 & 9 & 4 & 6 & 5 \\ \hline
    \cellcolor{lightgray}2 & \cellcolor{lightgray}3 & \cellcolor{lightgray}1 & 9 & 6 & 4 & 8 & 5 & 7 \\ \hline
    7 & 8 & 6 & \cellcolor{lightgray!80}4 & \cellcolor{lightgray!80}5 & 3 & 9 & 1 & 2 \\ \hline
    8 & 9 & 7 & \cellcolor{lightgray!80}5 & \cellcolor{lightgray!80}4 & 1 & 3 & 2 & 6 \\ \hline
    9 & 5 & 8 & 2 & 1 & \cellcolor{lightgray!60}6 & \cellcolor{lightgray!60}7 & 4 & 3 \\ \hline
    5 & 4 & 9 & 8 & 2 & \cellcolor{lightgray!60}7 & \cellcolor{lightgray!60}6 & 3 & 1 \\ \hline
    4 & 6 & 5 & 3 & 7 & 2 & 1 & \cellcolor{lightgray!40}8 & \cellcolor{lightgray!40}9 \\ \hline
    6 & 7 & 4 & 1 & 3 & 5 & 2 & \cellcolor{lightgray!40}9 & \cellcolor{lightgray!40}8 \\ \hline
    \end{array}$
        \caption{An order 9 latin square with four disjoint subsquares}
        \label{fig: square example}
    \end{subfigure}
    \caption{Latin squares with disjoint subsquares}
    \end{figure}
    
\end{example}

Given orders $h_1,\dots, h_k$, a latin square with pairwise disjoint subsquares of those orders is known as an \emph{incomplete latin square} (ILS). In the case when $(h_1\dots h_k)$ is a partition of $n$, so $\sum_{i=1}^kh_i = n$, such a latin square is known as a \emph{realization}, a \emph{partitioned incomplete latin square} (PILS) or a \emph{holey transversal design} (HTD) of block size 3. The latin square in \Cref{fig: square example} is a realization of $(3,2,2,2)$.

The existence of a realization, denoted $\LS(h_1\dots h_k)$, is equivalent to a question originally asked by L. Fuchs \cite{keedwell2015latin} about quasigroups with disjoint subquasigroups. It is a partially solved problem, with much of the work done for families of partitions with either a limited number of parts or of distinct integers. The same question has also been asked for latin cubes and hypercubes, where an $m$-realization is an $m$-dimensional hypercube with $m$-dimensional pairwise disjoint subhypercubes \cite{donovan2025latin}.

Unless otherwise stated, we assume that the parts of the partition $(h_1\dots h_k)$ are written in non-increasing order and that $(h_1^{\alpha_1}h_2^{\alpha_2}\dots h_m^{\alpha_m})$ represents a partition with $\alpha_j$ copies of $h_j$. Also, we assume that all realizations are in \emph{normal form}, where the subsquares appear along the main diagonal and the $i^{th}$ subsquare is of order $h_i$. \Cref{fig: square example} shows a realization in normal form. Note that for any integer $a$, the set $a+[n] = \{a+1,a+2,\dots,a+n\}$.

A \emph{transversal} in a latin square of order $n$ is a set of $n$ cells, with all cells in distinct rows and columns, and containing distinct symbols. Two transversals in a latin square are disjoint if they do not share any cells. It has been shown that for $n\neq 2,6$ there exists a latin square of order $n$ with $n$ pairwise disjoint transversals, which is equivalent to a pair of orthogonal latin squares (see \cite{bose1960further}).

\subsection{Known results}

The work done on this problem has generally followed two approaches. The results usually limit either the number of subsquares or the number of distinct part sizes in a partition. In this work, we construct realizations following the second approach, but we begin here by giving a summary of the major known results.

For realizations with at most four subsquares, existence has been determined by Heinrich \cite{heinrich2006latin}.

\begin{theorem}
\label{thm: small n squares}
    Take a partition $(h_1h_2\dots h_k)$ of $n$ with $h_1\geq h_2\geq \dots\geq h_k > 0$. Then an $\LS(h_1h_2\dots h_k)$
    \begin{itemize}
        \item always exists when $k=1$;
        \item never exists when $k=2$;
        \item exists when $k=3$ if and only if $h_1=h_2=h_3$;
        \item exists when $k=4$ if and only if $h_1=h_2=h_3$, or $h_2=h_3=h_4$ with $h_1\leq 2h_4$.
    \end{itemize}
\end{theorem}

Existence has been further determined for realizations with five subsquares in \cite{kemp2024latin}, and the conditions on existence are more complex than the smaller cases.

\begin{theorem}
\label{thm: 5 subsquares}
    An $\LS(h_1\dots h_5)$ exists if and only if 
    \begin{equation*}
        n^2 - \sum_{i=1}^5h_i^2\geq 3\left( \sum_{i\in D}h_i\right)\left(\sum_{j\in\overline{D}}h_j\right)
    \end{equation*}
    for all subsets $D\subseteq[5]$ where $|D| = 3$.
\end{theorem}

\Cref{thm: a^n square} and \Cref{thm: squaresatmost2} follow the second approach, and consider partitions with at most one and two distinct part sizes respectively. The first result is from D{\'e}nes and P{\'a}sztor \cite{denes1963some}. For the cases which are not covered in \Cref{thm: small n squares}, the results of Heinrich \cite{heinrich1982disjoint} and Kuhl et al.~\cite{kuhl2018latin} have been combined into \Cref{thm: squaresatmost2}.

\begin{theorem}
\label{thm: a^n square}
    For $k\geq 1$ and $a\geq 1$, an $\LS(a^k)$ exists if and only if $k\neq 2$.
\end{theorem}

\begin{theorem}
\label{thm: squaresatmost2}
    For $a>b>0$ and $k>4$, an $\LS(a^ub^{k-u})$ exists if and only if $u\geq 3$, or $0< u < 3$ and $a\leq (k-2)b$.
\end{theorem}

Two conjectures are given by Colbourn in \cite{colbourn2018latin}, each concerning a family of partitions for which realizations are believed to exist unconditionally.

\begin{conjecture}
\label{conj: largest 3}
    If $k\geq 3$, then an $\LS(h_1^3h_4\dots h_k)$ exists.
\end{conjecture}

The realizations in this first conjecture all have the largest three subsquares of the same size. \Cref{thm: small n squares} covers the cases when $k=3$ or $k=4$, and $k\in\{5,6\}$ was proven in \cite{colbourn2018latin}.

\begin{conjecture}
\label{conj: k-2}
    If $k\geq 5$ and $0<h_1\leq (k-2)h_k$, then an $\LS(h_1\dots h_k)$ exists.
\end{conjecture}

The following weaker result was proven in \cite{colbourn2018latin} with some possible exceptions, which were constructed in \cite{kemp2024latin}.

\begin{theorem}
    If $k\geq 5$, $h_k>0$ and $h_1\leq 3h_k$, then an $\LS(h_1\dots h_k)$ exists.
\end{theorem}

\section{Outline squares}

Outline squares have been used in \cite{kemp2024latin,kuhl2019existence,kuhl2018latin} to construct realizations for a range of partitions. Outline squares and outline rectangles were introduced by Hilton in \cite{hilton1980reconstruction}, where outline rectangles are a generalisation of outline squares.

\begin{definition}
    Given a partition $P$ of $n$, where $P = (p_1\dots p_k)$, let $O$ be a $k\times k$ array of multisets, with elements from $[k]$. For $i,j\in [k]$, let $O(i,j)$ be the multiset of symbols in cell $(i,j)$ and let $|O(i,j)|$ be the number of symbols in the cell, including repetition.
    
    Then $O$ is an \emph{outline square} associated to $P$ if
    \begin{enumerate}[label=\textup{(\arabic*)}]
        \item $|O(i,j)| = p_ip_j$, for all $i,j\in[k]$;
        \item symbol $\ell\in[k]$ occurs $p_ip_\ell$ times in the row $(i,[k])$;
        \item symbol $\ell\in[k]$ occurs $p_jp_\ell$ times in the column $([k],j)$.
    \end{enumerate}
\end{definition}

Let $O_\ell(i,j)$ represent the number of copies of symbol $\ell$ in the multiset $O(i,j)$. It is obvious that each value $O_\ell(i,j)$ must be a non-negative integer. In \cite{kemp2024latin}, the idea of a \emph{rational outline square} was introduced, where each of these values can instead be a non-negative rational number.

\begin{definition}
    Let $O$ be the set of values $\{O_\ell(i,j)\mid i,j,\ell\in[k]\}$, where $O_\ell(i,j)$ is a non-negative rational number for all $i,j,\ell\in[k]$. Then $O$ forms a \emph{rational outline square} associated to a partition $P = (p_1\dots p_k)$ if
    \begin{enumerate}[label=\textup{(\arabic*)}]
        \item $\sum_{\ell\in[k]} O_\ell(i,j) = p_ip_j$, for all $i,j\in[k]$;
        \item $\sum_{j\in[k]} O_\ell(i,j) = p_ip_\ell$, for all $i,\ell\in[k]$;
        \item $\sum_{i\in[k]} O_\ell(i,j) = p_jp_\ell$, for all $j,\ell\in[k]$.
    \end{enumerate}
\end{definition}

\begin{example}
\label{ex: outline squares}

The array in \Cref{outline square} is an outline square associated to $P = (3^12^3)$. \Cref{rational outline rectangle} gives a rational outline square for the same partition, where $\ell:x$ in cell $(i,j)$ represents that $O_\ell(i,j)=x$.

    \begin{figure}[h]
    \centering
    \begin{subfigure}{0.38\textwidth}
        \centering
        $\arraycolsep=4pt\begin{array}{|ccc|cc|cc|cc|} \hline
        1 & 1 & 1 & 3 & 3 & 2 & 2 & 2 & 2 \\
        1 & 1 & 1 & 3 & 4 & 2 & 4 & 2 & 3 \\
        1 & 1 & 1 & 4 & 4 & 4 & 4 & 3 & 3 \\ \hline
        3 & 3 & 3 & 2 & 2 & 1 & 1 & 1 & 1 \\ 
        4 & 4 & 4 & 2 & 2 & 1 & 4 & 1 & 3 \\ \hline
        2 & 2 & 2 & 1 & 1 & 3 & 3 & 1 & 1 \\
        4 & 4 & 4 & 1 & 4 & 3 & 3 & 1 & 2 \\ \hline
        2 & 2 & 2 & 1 & 1 & 1 & 1 & 4 & 4 \\
        3 & 3 & 3 & 1 & 3 & 1 & 2 & 4 & 4 \\ \hline
        \end{array}$
        \caption{An outline square.}
        \label{outline square}
    \end{subfigure}
    \begin{subfigure}{0.38\textwidth}
        \centering
        $\arraycolsep=4pt\begin{array}{|c|c|c|c|} \hline
        \thead{1:9} & \thead{3:7/2\\ 4:5/2} & \thead{2:5/2\\ 4:7/2} & 
        \thead{2:7/2\\ 3:5/2} \\ \hline
        \thead{3:5/2\\ 4:7/2} & \thead{2:4} & \thead{1:7/2\\ 4:1/2} & 
        \thead{1:5/2\\ 3:3/2} \\ \hline
        \thead{2:7/2\\ 4:5/2} & \thead{1:5/2\\ 4:3/2} & \thead{3:4} & 
        \thead{1:7/2\\ 2:1/2} \\ \hline
        \thead{2:5/2\\ 3:7/2} & \thead{1:7/2\\ 3:1/2} & \thead{1:5/2\\ 2:3/2} & 
        \thead{4:4} \\ \hline
        \end{array}$
        \caption{A rational outline square.}
        \label{rational outline rectangle}
    \end{subfigure}
    \caption{Outline squares associated to $P = (3^12^3)$.}
    \end{figure}
\end{example}

Also introduced in \cite{kemp2024latin} is the idea of a \emph{symmetric} (rational) outline square, where $O_\ell(i,j) = O_c(a,b)$ for every permutation $(a,b,c)$ of $(i,j,\ell)$. It follows that the ordering of the arguments $i,j,\ell$ in $O_\ell(i,j)$ is not important, and so we use $O(i,j,\ell)$ to represent the value of $O_c(a,b)$, where $(a,b,c)$ is any permutation of $(i,j,\ell)$.

\begin{definition}
    Given a partition $P=(p_1\dots p_k)$, let $O = \{O(i,j,\ell)\mid \{i,j,\ell\} \text{ is a multiset of } [k]\}$ be a set of non-negative (rational numbers) integers. Then $O$ defines a symmetric (rational) outline square associated to $P$ if
    \begin{equation}
    \label{eq: symmetric outline square cond}
        \sum_{\ell\in [k]}O(i,j,\ell) = p_ip_j\text{ for all $i,j\in[k]$.}
    \end{equation}
\end{definition}

A (rational) outline square $O$ is said to \emph{respect} $P$ if $O$ is associated to $P$ and $O_i(i,i) = p_i^2$ for all $i\in[k]$. For a symmetric outline square, this implies that $O(i,i,i) = p_i^2$ and $O(i,i,j)=0$ for $j\neq i$.

The outline squares in \Cref{ex: outline squares} respect $P = (3^12^3)$, and the outline square in \Cref{outline square} is symmetric. To see how this relates to our problem of finding a realization, we return to latin squares.

    Given a latin square $L$ of order $n$, the \emph{reduction modulo $P$} of $L$, denoted $O$, is the $k\times k$ array of multisets obtained by amalgamating rows $(p_1 + \dots + p_{i-1}) + [p_i]$ for all $i\in[k]$, columns $(p_1 + \dots + p_{j-1}) + [p_j]$ for all $j\in[k]$, and symbols $(p_1 + \dots + p_{\ell-1}) + [p_\ell]$ for all $\ell\in[k]$.

    Thus, $O$ is a $k\times k$ array of multisets on the symbols of $[k]$. This array is clearly an outline square associated to $P$.

The outline square in \Cref{outline square} is a reduction modulo $P = (3^12^3)$ of the latin square in \Cref{fig: square example}.

If an outline square $O$ is a reduction modulo $P$ of a latin square $L$, then we say that $O$ \emph{lifts} to $L$.

It is clear that every reduction of a latin square is an outline square, and the following result by Hilton proves the reverse.

\begin{theorem}[\cite{hilton1980reconstruction}]
\label{thm: outline rectangle to square}
    Let $P$ be a partition of $n$. For every outline square $O$ associated to $P$, there is a latin square $L$ of order $n$ such that $O$ lifts to $L$.
\end{theorem}

For any partition $P$, this idea can be used to construct a realization of $P$ from an outline square $O$, if $O$ respects $P$.

\begin{lemma}[\cite{kuhl2019existence}]
\label{lemma: outline to realization}
    For a partition $P = (h_1\dots h_k)$, an outline square which respects $P$ lifts to a realization of $P$.
\end{lemma}

Taking \Cref{thm: outline rectangle to square} and \Cref{lemma: outline to realization} together, it is observed that the existence of an outline square which respects $P$ is sufficient to prove the existence of a realization of $P$.

With this connection between outline squares and realizations established, we now consider the relationship between realizations and symmetric rational outline squares.

Throughout, $\ROS(P)$ denotes a symmetric rational outline square which respects the partition $P$.

\begin{lemma}
\label{lemma: always symmetric solution}
    If an $\LS(h_1\dots h_k)$ exists, then an $\ROS(h_1\dots h_k)$ exists.
\end{lemma}
\begin{proof}
    Obtain an outline square $O$ which respects $P = (h_1\dots h_k)$ by reducing the $\LS(h_1\dots h_k)$ modulo $P$. Construct a symmetric rational outline square $X$ by:
    $$X(i,j,\ell) = \begin{cases}
        h_i^2, & \text{if $i=j=\ell$,}\\
        \frac{1}{6}[O_\ell(i,j) + O_\ell(j,i) + O_i(j,\ell) + O_i(\ell,j) + O_j(i,\ell) + O_j(\ell,i)], & \text{if $i$, $j$ and $\ell$ are distinct,}\\
        0, & \text{otherwise.}
    \end{cases}$$

    Clearly, $X(i,i,i) = h_i^2$ as required to respect the partition, and by fixing any $i,j\in[k]$ where $i\neq j$, we have that
    \begin{align*}
        \sum_{\ell\in[k]} X(i,j,\ell) &= \frac{1}{6}\left[ \sum_{\ell=1}^k O_\ell(i,j) + \sum_{\ell=1}^k O_\ell(j,i) + \sum_{\ell=1}^k O_i(j,\ell) + \sum_{\ell=1}^k O_i(\ell,j) + \sum_{\ell=1}^k O_j(i,\ell) + \sum_{\ell=1}^k O_j(\ell,i) \right]\\
        &= \frac{1}{6}[6h_ih_j] = h_ih_j.
    \end{align*}

    Thus, $X$ forms a $\ROS(h_1\dots h_k)$.
\end{proof}

Given a rational outline square, there must exist some sufficiently large integer $q$ such that multiplying all of the values $X(i,j,\ell)$ by $q$ gives only integer values. Thus, we obtain the following lemma.

\begin{lemma}
\label{lemma: always outline for multiple of P}
    If an $\ROS(h_1\dots h_k)$ exists, then there exists an integer $q$ such that a symmetric $\LS([qh_1]\dots[qh_k])$ exists.
\end{lemma}

Combining \Cref{lemma: always symmetric solution} with \Cref{lemma: always outline for multiple of P} gives the following result.

\begin{corollary}
    If an $\LS(h_1\dots h_k)$ exists, then there exists an integer $q$ such that a symmetric $\LS([qh_1]\dots[qh_k])$ exists.
\end{corollary}

\section{Necessary conditions}

It can be seen in \Cref{thm: small n squares} that the complexity of the necessary and sufficient conditions increases as the number of subsquares increases. Two necessary conditions were proven by Colbourn \cite{colbourn2018latin} for an arbitrary number of subsquares, and these are given as \Cref{condition: squarecondition1} and \Cref{condition: squarecondition2}. It was shown in \cite{kemp2024latin} that the second condition alone is sufficient for $k=5$, however, it was established in \cite{colbourn2018latin} that further necessary conditions are required for larger values of $k$.

\begin{theorem}[\cite{colbourn2018latin}]
\label{condition: squarecondition1}
    If an $\LS(h_1h_2\dots h_k)$ exists, then $h_1\leq\sum_{i=3}^kh_i$.
\end{theorem}

\begin{theorem}[\cite{colbourn2018latin}]
\label{condition: squarecondition2}
    If an $\LS(h_1h_2\dots h_k)$ exists, then $$n^2 - \sum_{i=1}^kh_i^2\geq 3\Bigg( \sum_{i\in D}h_i\Bigg)\Bigg(\sum_{j\in\overline{D}}h_j\Bigg)$$ for any $D\subseteq\{1,2,\dots,k\}$ and $\overline{D} = \{1,2,\dots,k\}\setminus D$.
\end{theorem}

We now use the existence of a rational outline square to improve the necessary conditions. In fact, we combine the known conditions into a single condition.

\begin{theorem}
\label{condition: sufiĉa?}
    If an $\LS(h_1\dots h_k)$ exists, then
    $$\left(\sum_{i\in A\cup C}h_i\right)^2 + \left(\sum_{i\in B\cup D}h_i\right)^2 - \sum_{i\in E}h_i^2 \geq \left(\sum_{i\in A\cup D}h_i\right)\left(\sum_{i\in B\cup C}h_i - \sum_{j\in \overline{E}}h_j\right),$$
    where $A$, $B$, $C$ and $D$ are pairwise disjoint subsets of $[k]$, $E = A\cup B\cup C\cup D$ and $\overline{E} = [k]\setminus E$.
\end{theorem}
\begin{proof}
    By \Cref{lemma: always symmetric solution}, if such a latin square exists, then an $\ROS(h_1\dots h_k)$ exists also. We consider this outline square $X$.

    Observe that there are $(\sum_{i\in A\cup D}h_i)(\sum_{i\in B\cup C}h_i)$ symbols, including repetition, required across the cells $(i,j)$ where $i\in A\cup D$ and $j\in B\cup C$. Of these symbols, at most $(\sum_{i\in A\cup D}h_i)(\sum_{j\in \overline{E}}h_j)$ can be symbols in $\overline{E}$. The remaining symbols must be from $E$.

    We now count the number of symbols from $E$ in these cells. First, note that there are at most $(\sum_{i\in A}h_i)^2 - \sum_{i\in A}h_i^2$ symbols from $A$ in the rows of $A$, and likewise for the rows of $D$ and columns of $B$ and $C$. Thus, we must now consider the symbols appearing not within columns or rows from the same subset. The number of symbols from $A$ within rows of $D$ of these cells is given by
    $$\sum_{a\in A} \sum_{d\in D} \sum_{i\in B\cup C} X(a,d,i)$$
    and the number of symbols from $D$ in columns of $A$ is the same by symmetry. Similarly, the number of symbols from $B$ in columns of $C$ and the number of symbols from $C$ in columns of $B$ is given by
    $$\sum_{b\in B} \sum_{c\in C} \sum_{i\in A\cup D} X(b,c,i).$$

    Combining these, and using symmetry to change the order of the terms in the $X$ variables, we see that
    \begin{align*}
        2\sum_{a\in A} \sum_{d\in D} \sum_{i\in B\cup C} X(a,d,i) + 2\sum_{b\in B} \sum_{c\in C} \sum_{i\in A\cup D} X(b,c,i) &= 2 \sum_{a\in A} \sum_{c\in C} \sum_{i\in B\cup D} X(a,c,i) + 2 \sum_{b\in B} \sum_{d\in D} \sum_{i\in A\cup C} X(b,d,i)\\
        &\leq 2 \sum_{a\in A} \sum_{c\in C} \sum_{i\in E} X(a,c,i) + 2 \sum_{b\in B} \sum_{d\in D} \sum_{i\in E} X(b,d,i).
    \end{align*}
    The two terms in this last expression are the number of symbols in cells $(i,j)$ where $i\in A$ and $j\in C$ or $i\in B$ and $j\in D$ respectively. Thus, we have that
    $$2 \sum_{a\in A} \sum_{d\in D} \sum_{i\in B\cup C} X(a,d,i) + 2 \sum_{b\in B} \sum_{c\in C} \sum_{i\in A\cup D} X(b,c,i) \leq 2(\sum_{i\in A}h_i)(\sum_{i\in C}h_i) + 2(\sum_{i\in B}h_i)(\sum_{i\in D}h_i).$$

    Altogether, the number of possible entries is at least the number of required symbols, so
    \begin{align*}
    \begin{split}
        \left(\sum_{i\in A\cup D}h_i\right)\left(\sum_{i\in B\cup C}h_i - \sum_{j\in \overline{E}}h_j\right)\leq \left(\sum_{i\in A}h_i\right)^2 + \left(\sum_{i\in B}h_i\right)^2 + \left(\sum_{i\in C}h_i\right)^2 + \left(\sum_{i\in D}h_i\right)^2 - \sum_{i\in E}h_i^2\\ \phantom{1} + 2\left(\sum_{i\in A}h_i\right)\left(\sum_{i\in C}h_i\right) + 2\left(\sum_{i\in B}h_i\right)\left(\sum_{i\in D}h_i\right)
    \end{split}
    \end{align*}
    $$\left(\sum_{i\in A\cup D}h_i\right)\left(\sum_{i\in B\cup C}h_i - \sum_{j\in \overline{E}}h_j\right)\leq \left(\sum_{i\in A\cup C}h_i\right)^2 + \left(\sum_{i\in B\cup D}h_i\right)^2 - \sum_{i\in E}h_i^2.$$
\end{proof}

    The known necessary conditions are recovered from this new condition by setting $C = D = \emptyset$. Taking $A = \{2\}$ and $B = \{1\}$ gives the condition in \Cref{condition: squarecondition1}, and letting $A\cup B = E = [k]$ returns \Cref{condition: squarecondition2}.

Both additions that are made in this condition (taking $E$ to be a proper subset of $[k]$, and letting $C$ and $D$ be non-empty) are required parts of any general necessary condition. If $E\subset [k]$ but $C$ and $D$ are allowed to be empty, then partitions such as $(20^1 19^1 10^1 8^1 1^4)$ satisfy the necessary condition. However, with $C$ and $D$ non-empty, this same partition does not meet the condition. This partition also demonstrates how the previous necessary conditions are not sufficient for $k\geq 8$.

\section{Rational solutions to two conjectures}

In this section, we construct symmetric rational outline squares for all partitions in \Cref{conj: largest 3,conj: k-2}. First, we introduce a result which constructs a symmetric rational outline square for a partition which falls between two similar partitions.

\begin{lemma}
\label{lemma: exists in range}

    For $a,b\geq 0$, allow the parts of the partitions $P_1 = (h_1\dots h_{m-1} (h_m-a) h_{m+1}\dots h_k)$ and $P_2 = (h_1\dots h_{m-1} (h_m+b) h_{m+1}\dots h_k)$ to not be in non-increasing order.

    If an $\ROS(P_1)$ and an $\ROS(P_2)$ exist, then an $\ROS(h_1\dots h_{m-1} h_m h_{m+1}\dots h_k)$ exists also.
\end{lemma}
\begin{proof}

    If $a=0$ or $b=0$, then the $\ROS(P_1)$, or $\ROS(P_2)$ respectively, is already an $\ROS(h_1\dots h_{m-1} h_m h_{m+1}\dots h_k)$. Suppose that $a,b\geq 1$.

    Denote the $\ROS(h_1\dots h_{m-1} (h_m-a) h_{m+1}\dots h_k)$ and $\ROS(h_1\dots h_{m-1} (h_m+b) h_{m+1}\dots h_k)$ as $A$ and $B$ respectively. Construct a new symmetric rational outline square $C$ respecting $(h_1\dots h_k)$, by taking $$C(i,j,\ell) =\begin{cases}
        h_m^2, & \text{if $i=j=\ell=m$,}\\
        \frac{b}{a+b}\cdot A(i,j,\ell) + \frac{a}{a+b}\cdot B(i,j,\ell), & \text{otherwise.}
    \end{cases}$$

    Then for $i\neq m$ and $j\neq m$, it is clear that $\sum_{\ell\in[k]} C(i,j,\ell) = \frac{b}{a+b}\sum_{\ell=1}^k A(i,j,\ell) + \frac{a}{a+b}\sum_{\ell=1}^k B(i,j,\ell) = \frac{a+b}{a+b}h_ih_j = h_ih_j$.

    Also, for $j\neq m$, we get that $\sum_{\ell\in[k]} C(m,j,\ell) = \frac{b}{a+b}\sum_{\ell=1}^k A(m,j,\ell) + \frac{a}{a+b}\sum_{\ell=1}^k B(m,j,\ell) = \frac{b}{a+b}h_j(h_m-a) + \frac{a}{a+b}h_j(h_m+b) = h_jh_m$.

    Since $A$ and $B$ must respect their respective partitions, $A(i,i,i) = h_i^2$ and $B(i,i,i) = h_i^2$ for all $i\in[k]\setminus\{m\}$. Thus, $C$ also has this property for all $i\neq m$. Also, $A(m,m,i) = B(m,m,i) = 0$ for all $i\neq m$, so $C(m,m,i) = 0$ also. It follows that $\sum_{\ell\in[k]} C(m,m,\ell) = h_m^2$. Thus, $C$ is a symmetric rational outline square respecting $(h_1\dots h_k)$.
\end{proof}

Using \Cref{lemma: always symmetric solution}, we obtain the following corollary.

\begin{corollary}
\label{corollary: exists in range}
    With $P_1$ and $P_2$ as defined above, if an $\LS(P_1)$ and an $\LS(P_2)$ exist, then there also exists an $\ROS(h_1\dots h_{m-1} h_m h_{m+1}\dots h_k)$.
\end{corollary}

Given a large set of partitions which can be realized, the method given in the previous lemma can be repeatedly applied to obtain outline squares that respect partitions.

\begin{lemma}
\label{lemma: All comb between range}
    If an $\LS(h_1^ih_k^{k-i})$ exists for all $u\leq i\leq k-1$, then an $\ROS(h_1^u g_1\dots g_{k-u-1} h_k)$ exists for any partition $(g_1\dots g_{k-u-1})$ with $h_1\geq g_1\geq\dots\geq g_{k-u-1}\geq h_k$.
\end{lemma}
\begin{proof}
    For all $u\leq i\leq k-2$, using \Cref{corollary: exists in range} with the partitions $(h_1^{i}h^{k-i})$ and $(h_1^{i+1}h_k^{k-i-1})$, for which realizations exist by assumption, it is shown that an $\ROS(h_1^ig_1h_k^{k-i-1})$ exists for all $h_1\geq g_1\geq h_k$.

    Consider some $a\in[k-u-2]$ and any partition $(g_1\dots g_a)$ with $h_1\geq g_\ell\geq h_k$ for all $\ell\in [a]$. Suppose that for all $u\leq i\leq k-a-1$ there exists an $\ROS(h_1^i g_1\dots g_a h_k^{k-i-a})$. Then for any $u\leq j\leq k-a-2$, by permuting the parts of the partitions, there exists an $\ROS(h_1^j g_1\dots g_a h_k^{k-j-a})$ and an $\ROS(h_1^j g_1\dots g_a h_1 h_k^{k-j-a-1})$. Using \Cref{lemma: exists in range}, there exists an $\ROS(h_1^j g_1\dots g_a g_{a+1} h_k^{k-j-a-1})$ for all partitions $(g_1\dots g_{a+1})$ with $h_1\geq g_{a+1}\geq h_k$.

    The result follows by induction on $a$.
\end{proof}

We now construct symmetric rational outline squares for all cases of \Cref{conj: k-2} and \Cref{conj: largest 3}. This does not prove the conjectures of course, but considering \Cref{lemma: always outline for multiple of P}, it does show that there exists a realization for some sufficiently large multiple of each partition.

\begin{theorem}
    For any $k\geq 5$, if $h_1\leq (k-2)h_k$ then an $\ROS(h_1\dots h_k)$ exists.
\end{theorem}
\begin{proof}
    Since $h_1\leq (k-2)h_k$, by \Cref{thm: a^n square} and \Cref{thm: squaresatmost2}, an $\LS(h_1^ih_k^{k-i})$ exists for all $1\leq i\leq k$. Thus by \Cref{lemma: All comb between range}, the result follows.
\end{proof}

\begin{theorem}
    If $h_1=h_2=h_3$, then an $\ROS(h_1^3h_4\dots h_k)$ exists for all $k\geq 3$.
\end{theorem}
\begin{proof}
    For $k=3$ and $k=4$, \Cref{thm: small n squares} gives that an $\LS(h_1^3)$, an $\LS(h_1^4)$ and an $\LS(h_1^3h_k)$ exist. For $k\geq 5$, an $\LS(h_1^ih_k^{k-i})$ exists for all $3\leq i\leq k$ by \Cref{thm: a^n square} and \Cref{thm: squaresatmost2}. In either case, use \Cref{lemma: All comb between range}.
\end{proof}

\section{Similar outline squares}

This section introduces a new property for outline squares which makes it easier to construct symmetric rational outline squares for realizations with many subsquares of the same size.

\begin{definition}
    Let $X$ be an $\ROS(h_1\dots h_a^m h_{a+m}\dots h_k)$ and let $M = a-1+[m]$. Then $X$ is \emph{similar with respect to $h_a^m$} if there exists non-negative values $X'(i,j,a)$, $X'(i,a,a)$ and $X'(a,a,a)$ for all $i,j\notin M$, such that for all distinct $\alpha,\beta,\gamma\in M$
\begin{enumerate}[label=\textup{(\arabic*)}]
    \item $X(i,j,\alpha) = X'(i,j,a)$,
    \item $X(i,\alpha,\beta) = X'(i,a,a)$, and
    \item $X(\alpha,\beta,\gamma) = X'(a,a,a)$.
\end{enumerate}
\end{definition}

    With these variables, the entire outline square is determined by only $X(i,j,\ell)$, $X'(i,j,a)$, $X'(i,a,a)$ and $X'(a,a,a)$ for all distinct $i,j,\ell\notin M$.

\begin{example}
    The outline square in \Cref{outline square} is a similar with respect to $2^3$, with $X'(1,2,2) = 3$ and $X'(2,2,2) = 1$.
\end{example}

\sloppy We simplify the notation by letting a similar symmetric rational outline square be denoted by $\SROS(h_1\dots h_k,\{h_{\alpha_1}^{m_1},\dots,h_{\alpha_\ell}^{m_\ell}\})$, where the outline square is similar with respect to each $h_{\alpha_i}^{m_i}$ separately. As in that case, we prove that such an array can always be constructed from a realization or a symmetric rational outline square.

\begin{lemma}
    An $\ROS(h_1\dots h_a^m h_{a+m}\dots h_k)$ exists if and only if an $\SROS(h_1\dots h_a^m h_{a+m}\dots h_k,\{h_a^m\})$ exists.
\end{lemma}
\begin{proof}
    Let $X$ be an $\ROS(h_1\dots h_k)$ and let $M = a-1+[m]$. Then, construct a new $\ROS(h_1\dots h_k)$, $Y$, as follows.
    For $i,j,\ell\in [k]$, let
    $$Y(i,j,\ell) = \begin{cases}
        h_i^2, & \text{if $i=j=\ell$,}\\
        X(i,j,\ell), & \text{if $i,j,\ell\notin M$ are distinct,}\\
        Y'(i,j,a), & \text{if $i\neq j$, $i,j\notin M$ and $\ell\in M$,}\\
        Y'(i,a,a), & \text{if $j\neq \ell$, $i\notin M$ and $j,\ell\in M$,}\\
        Y'(a,a,a), & \text{if $i$, $j$ and $\ell$ are distinct, and $i,j,\ell\in M$,}\\
        0, & \text{otherwise,}
    \end{cases}$$
    where, for distinct $i,j\notin M$,
    $$Y'(i,j,a) = \frac{1}{m}\sum_{\alpha\in M} X(i,j,\alpha),$$
    $$Y'(i,a,a) = \frac{1}{m(m-1)}\sum_{\beta\in M}\sum_{\alpha\in M} X(i,\beta,\alpha),$$
    $$\text{and } Y'(a,a,a) = \frac{1}{m(m-1)(m-2)}\left(-mh_a^2+\sum_{\gamma\in M}\sum_{\beta\in M}\sum_{\alpha\in M} X(\gamma,\beta,\alpha)\right).$$

    It is straightforward to check that $Y$ forms a symmetric rational outline square respecting $(h_1\dots h_k)$.
\end{proof}

This method of constructing a similar outline square can be repeated to give an outline square which is similar in respect to more parts, since the construction preserves any existing similarity with respect to other parts of the partition.

\begin{corollary}
\label{cor: ROS iff SROS}
    An $\ROS(h_1\dots h_{\alpha_1}^{m_1}\dots h_{\alpha_\ell}^{m_\ell}\dots h_k)$ exists if and only if an $\SROS(h_1\dots h_k,\{h_{\alpha_1}^{m_1},\dots,h_{\alpha_\ell}^{m_\ell}\})$ exists.
\end{corollary}

Applying \Cref{lemma: always symmetric solution}, this can of course be extended to realizations.

\begin{corollary}
    If an $\LS(h_1\dots h_{\alpha_1}^{m_1}\dots h_{\alpha_\ell}^{m_\ell}\dots h_k)$ exists, then an $\SROS(h_1\dots h_k,\{h_{\alpha_1}^{m_1},\dots,h_{\alpha_\ell}^{m_\ell}\})$ exists.
\end{corollary}

The similarity property reduces the number of free variables when constructing an outline square. This may simplify the construction. In fact, as the following lemma shows, we need only determine the values $X(i,j,\ell)$ for distinct $i,j,\ell\notin M$.

\begin{lemma}
\label{lemma: similar only needs distinct triples}
    Let $X$ be an $\SROS(h_1\dots h_a^m h_{a+m}\dots h_k, \{h_a^m\})$ with $m\geq 3$, and let $M = a-1+[m]$. Then
    \begin{enumerate}[label=\textup{(\arabic*)}]
        \item $X'(i,j,a) = \displaystyle\frac{1}{m}\left(h_ih_j - \sum_{\ell\notin M}X(i,j,\ell)\right)$,
        \item $\displaystyle X'(i,a,a) = \frac{1}{m-1}\left(h_ih_a - \frac{1}{m}\sum_{\ell\notin M\cup\{i\}} h_ih_\ell + \frac{2}{m}\sum_{\substack{j,\ell\notin M\cup\{i\}\\ j<\ell}} X(i,j,\ell)\right)$,
        \item $X'(a,a,a) = \displaystyle\frac{1}{m(m-1)(m-2)}\left(m(m-1)h_a^2 - mh_a\sum_{i\notin M}h_i + 2\sum_{\substack{i,j\notin M\\ i<j}}h_ih_j - 6\sum_{\substack{i,j,\ell\notin M\\ i<j<\ell}}X(i,j,\ell)\right)$.
    \end{enumerate}
\end{lemma}
\begin{proof}
    By definition of a symmetric rational outline square, for distinct $i,j\notin M$, $h_ih_j = \sum_{\ell\in[k]} X(i,j,\ell)$ and so
    $$h_ih_j = \sum_{\ell=1}^k X(i,j,\ell) = \sum_{\ell\notin M}X(i,j,\ell) + \sum_{\alpha\in M}X(i,j,\alpha) = \sum_{\ell\notin M}X(i,j,\ell) + mX'(i,j,a).$$
    Thus, $X'(i,j,a) = \displaystyle\frac{1}{m}\left(h_ih_j - \sum_{\ell\notin M}X(i,j,\ell)\right)$.

    In a similar manner, we consider $i\notin M$ and $\alpha\in M$. Then
    $$h_ih_a = \sum_{\ell\notin M}X(i,\ell,\alpha) + \sum_{\beta\in M\setminus\{\alpha\}}X(i,\beta,\alpha) = \sum_{\ell\notin M\cup\{i\}}X'(i,\ell,a) + (m-1)X'(i,a,a).$$ Therefore, $\displaystyle X'(i,a,a) = \frac{1}{m-1}\left(h_ih_a - \frac{1}{m}\sum_{\ell\notin M\cup\{i\}} h_ih_\ell + \frac{2}{m}\sum_{\substack{j,\ell\notin M\cup\{i\}\\ j<\ell}} X(i,j,\ell)\right)$.

    Finally, for distinct $\alpha,\beta\in M$, $$h_a^2 = \sum_{\ell\notin M}X(\ell,\alpha,\beta) + \sum_{\gamma\in M\setminus\{\alpha,\beta\}}X(\alpha,\beta,\gamma) = \sum_{\ell\notin M} X'(\ell,a,a) + (m-2)X'(a,a,a).$$
    Thus, $X'(a,a,a) = \displaystyle\frac{1}{m(m-1)(m-2)}\left(m(m-1)h_a^2 - mh_a\sum_{i\notin M}h_i + 2\sum_{\substack{i,j\notin M\\ i<j}}h_ih_j - 6\sum_{\substack{i,j,\ell\notin M\\ i<j<\ell}}X(i,j,\ell)\right)$.
\end{proof}

In the following section we use this result to construct similar symmetric rational outline squares for a few different families of partitions.

\section{Partitions with many parts of the same size}

This section focuses on constructions of similar symmetric rational outline squares for different families of partitions: $(h_1h_2h_3^m)$, $(h_1h_2h_3h_4^m)$ and $(h_1h_2^ah_3^b)$, where the parts are not necessarily in non-increasing order.

If $m\leq 2$, then the existence of an $\LS(h_1h_2h_3^m)$ or an $\LS(h_1h_2h_3h_4^m)$ has already been determined in \Cref{thm: small n squares,thm: 5 subsquares}. In the third case, if either $a<3$ or $b<3$, then the partition is in the form of $(h_1h_2h_3^m)$ or $(h_1h_2h_3h_4^m)$. Thus, we only consider $m,a,b\geq 3$. 

Rational outline squares can be constructed fairly easily for the first two families by using \Cref{lemma: similar only needs distinct triples}.

\begin{theorem}
\label{thm: h_1h_2h_3^m rational}
    A $\ROS(h_1h_2h_3^m)$ exists for $m\geq 3$ if and only if $h_1\leq mh_3$, $h_2\leq mh_3$, and $h_3^2 - \frac{h_3}{m-1}(h_1+h_2) + \frac{2}{m(m-1)}h_1h_2 \geq 0$.
\end{theorem}
\begin{proof}
    All necessary conditions come from \Cref{condition: sufiĉa?}. For each condition, set $C = D = \emptyset$ and take $A = \{2\}$ and $B = \{1\}$, $A = \{1\}$ and $B = \{2\}$, or $A = \{1,2\}$ and $B = [k]\setminus A$.
    
    By \Cref{lemma: similar only needs distinct triples}, when constructing $X$ to be an $\SROS(h_1h_2h_3^m,\{h_3^m\})$, we need only consider the variables $X(i,j,\ell)$ where $i$, $j$ and $\ell$ are distinct values from $\{1,2\}$. As there are no such variables, the solution is predetermined.

    Thus, $X'(1,2,3) = \frac{1}{m}h_1h_2$, $X'(1,3,3) = \frac{1}{m-1}(h_1h_3 - \frac{1}{m}h_1h_2)$, $X'(2,3,3) = \frac{1}{m-1}(h_2h_3 - \frac{1}{m}h_1h_2)$, and $X'(3,3,3) = \frac{1}{m-2}[h_3^2 - \frac{1}{m-1}h_3(h_1+h_2) + \frac{2}{m(m-1)}h_1h_2]$. When the assumed inequalities are satisfied, these are all non-negative.
\end{proof}

\begin{theorem}
    Take $h_1\geq h_2\geq h_3$. For $m\geq 3$, an $\ROS(h_1 h_2 h_3 h_4^m)$ exists if and only if the following are satisfied:
    \begin{enumerate}[label=\textup{(\arabic*)}]
        \item $n^2 - \sum_{i=1}^k h_i^2\geq 3mh_4(h_1+h_2+h_3)$,
        \item $n^2 - \sum_{i=1}^k h_i^2\geq 3(mh_4 + h_j)(n-mh_4-h_j)$ for all $j\in[3]$,
        \item $h_1(h_2+h_3)\leq mh_1h_4 + 2h_2h_3$, and
        \item $h_4 \geq \frac{1}{m}(h_1-h_3)$.
    \end{enumerate}
\end{theorem}
\begin{proof}
    The four conditions are shown to be necessary using \Cref{condition: sufiĉa?}. For each condition respectively, set $C=D=\emptyset$ and use
    \begin{enumerate}[label=\textup{(\arabic*)}]
        \item $A = \{1,2,3\}$, $B = [k]\setminus A$,
        \item $A = \{1,2,3\}\setminus\{j\}$, $B = [k]\setminus A$,
        \item $A = \{1\}$, $B = \{2,3\}$,
        \item $A = \{2\}$, $B = \{1\}$.
    \end{enumerate}

    Suppose that the conditions are met. We construct an $\SROS(h_1h_2h_3h_4^m,\{h_4^m\})$, $X$, as follows:

    Let $X(1,2,3) = \min\{h_2h_3,\frac{1}{6}[n^2 - \sum_{i=1}^k h_i^2 - 3mh_4(h_1+h_2+h_3)]\}$. Then by \Cref{lemma: similar only needs distinct triples}, the remaining values of $X$ are determined also. Now, we must check that all values are non-negative.

    $X(1,2,3)\geq 0$ from (1).
    Observe that for distinct $i,j\in [3]$, $X'(i,j,4) = \frac{1}{m}(h_ih_j - X(1,2,3))$. Since $h_2h_3\leq h_ih_j$ for all distinct $i,j\in[3]$, we know that $X'(i,j,4)\geq 0$. Similarly, $X'(4,4,4) = \frac{1}{m(m-1)(m-2)}[n^2 - \sum_{i=1}^k h_i^2 - 3mh_4(h_1+h_2+h_3) - 6X(1,2,3)]$ and so $X'(4,4,4)\geq 0$.

    The final value to check is $X'(i,4,4)$ for $i\in[3]$. 
    If $X(1,2,3) = h_2h_3$, then $X'(i,4,4) = \frac{1}{m(m-1)}[mh_ih_4 - \sum_{j\in[3]}h_ih_j + h_i^2 +2h_2h_3]$. For $i=2$, $mh_4 - h_1 + h_3 \geq 0$ from (4), and so $X'(2,4,4)\geq 0$. For $i=3$, $mh_4 - h_1 + h_2 \geq mh_4 - h_1 + h_3 \geq 0$. Thus, $X'(3,4,4)\geq 0$. When $i=1$, $mh_1h_4 - h_1h_2 - h_1h_3 + 2h_2h_3\geq 0$ from (3).

    Instead consider the case where $X(1,2,3) = \frac{1}{6}[n^2 - \sum_{i=1}^k h_i^2 - 3mh_4(h_1+h_2+h_3)]$. With some rearrangement, it can be seen that $X'(i,4,4) = \frac{1}{3m(m-1)}[n^2 - \sum_{i=1}^k h_i^2 - 3(mh_4+h_i)(h_1+h_2+h_3-h_i)]\geq 0$ using (2).

    Therefore, for either case of $X(1,2,3)$, all values are non-negative, and $X$ is a rational outline square that respects the partition.
\end{proof}

The third family of partitions requires a more complex construction, as we use an outline square that is similar with respect to two different sizes of subsquare. Thus, we cannot utilise \Cref{lemma: similar only needs distinct triples}.

\begin{lemma}
\label{lemma: h_1h_2^ah_3^b inequalities}
    For $a,b\geq 3$, an $\SROS(h_1 h_2^a h_3^b, \{h_2^a,h_3^b\})$ exists if and only if there exists rational values $X'(2,2,3)$ and $X'(2,3,3)$ satisfying the following inequalities:
    \begin{enumerate}[label=\textup{(\arabic*)}]
        \item $X'(2,2,3),X'(2,3,3)\geq 0$,
        \item $abh_2h_3\geq ab(a-1)X'(2,2,3) + ab(b-1)X'(2,3,3)\geq \max\{abh_2h_3-ah_1h_2,abh_2h_3-bh_1h_3\}$,
        \item $2ab(a-1)X'(2,2,3) + ab(b-1)X'(2,3,3)\leq a(a-1)h_2^2 - ah_1h_2 + abh_2h_3$, and
        \item $ab(a-1)X'(2,2,3) + 2ab(b-1)X'(2,3,3)\leq b(b-1)h_3^2 - bh_1h_3 + abh_2h_3$.
    \end{enumerate}
\end{lemma}
\begin{proof}
    Rearrange the partition to $(h_1h_2h_3h_2^{a-1}h_3^{b-1})$, and let $M_2 = \{2\}\cup (3+[a-1])$ and $M_3 = \{3\}\cup (a+2+[b-1])$. Suppose that $X$ is an $\SROS(h_1h_2h_3 h_2^{a-1} h_3^{b-1}, \{h_2^a,h_3^b\})$. Then $X$ has only seven variables to determine the entire outline square: $X'(1,2,3)$, $X'(1,2,2)$, $X'(1,3,3)$, $X'(2,2,2)$, $X'(3,3,3)$, $X'(2,2,3)$ and $X'(2,3,3)$. We find each of these in terms of the last two.
    To satisfy \Cref{eq: symmetric outline square cond} with $i=2$ and $j=3$,
    \begin{align*}
        h_2h_3 &= \sum_{\ell\in[a+b+1]}X(2,3,\ell)\\
        &= X(1,2,3) + \sum_{\ell\in M_2}X(2,\ell,3) + \sum_{\ell\in M_3}X(2,3,\ell)\\
        &= X'(1,2,3) + (a-1)X'(2,2,3) + (b-1)X'(2,3,3)
    \end{align*}
    Since $X'(1,2,3)\geq 0$, it follows that $h_2h_3\geq (a-1)X'(2,2,3) + (b-1)X'(2,3,3)$, which gives one side of (2).

    The remaining four variables of $X'$ can be written in terms of $X'(2,2,3)$ and $X'(2,3,3)$ by considering \Cref{eq: symmetric outline square cond} with each of the $(i,j)$ pairs $(1,2)$, $(1,3)$, $(2,2)$ and $(3,3)$. This results in the following
    \begin{itemize}
        \item $abX'(1,2,3) = abh_2h_3 - ab(a-1)X'(2,2,3) - ab(b-1)X'(2,3,3)$,
        \item $a(a-1)X'(1,2,2) = ah_1h_2 - abh_2h_3 + ab(a-1)X'(2,2,3) + ab(b-1)X'(2,3,3)$,
        \item $b(b-1)X'(1,3,3) = bh_1h_3 - abh_2h_3 + ab(a-1)X'(2,2,3) + ab(b-1)X'(2,3,3)$,
        \item $a(a-1)(a-2)X'(2,2,2) = a(a-1)h_2^2 - ah_1h_2 + abh_2h_3 - 2ab(a-1)X'(2,2,3) - ab(b-1)X'(2,3,3)$, and
        \item $b(b-1)(b-2)X'(3,3,3) = b(b-1)h_3^2 - bh_1h_3 + abh_2h_3 - ab(a-1)X'(2,2,3) - 2ab(b-1)X'(2,3,3)$,
    \end{itemize}
    and since each $X'$ variable must be non-negative, the inequalities (1)-(4) are obtained.
\end{proof}

\begin{lemma}
\label{lemma: h_1h_2^ah_3^b inequality system solution}
    Given the system of inequalities:
    \begin{enumerate}[label=\textup{(\arabic*)}]
        \item $2x+y\leq \alpha$
        \item $x+2y\leq \beta$
        \item $x+y\geq \gamma$
        \item $x+y\leq \delta$
    \end{enumerate}
    There is a non-negative solution to this system if and only if $\alpha,\beta,\delta\geq 0$, $\gamma\leq \alpha,\beta,\delta$ and $\alpha+\beta\geq 3\gamma$. Further, if the variables $\alpha$, $\beta$ and $\gamma$ are rational, then the solution is also rational.
\end{lemma}
\begin{proof}
    We begin by assuming that there is a solution to the system with $x,y\geq 0$. Then it is clear from $(1)$, $(2)$ and $(4)$ that $\alpha,\beta,\delta\geq 0$, and from $(3)$ and $(4)$ that $\gamma\leq \delta$. From $(1)$ and $(3)$, we know that $\gamma\leq x+y\leq 2x+y\leq \alpha$, and similarly we have that $\gamma\leq \beta$ from $(2)$ and $(3)$. Finally, combining $(1)$ and $(2)$, it is seen that $3\gamma\leq \alpha+\beta$.

    We now suppose that the conditions are met. Firstly, if $\gamma\leq 0$, then taking $x=y=0$ is a solution to the system. Thus, we assume that $\gamma>0$ from here.
    
    We break the solution into two cases.

    If $\beta\geq 2\gamma$, take $x=0$ and $y=\gamma$. Then it is obvious that $(3)$ and $(4)$ are satisfied. For $(1)$, $2x+y = \gamma\leq \alpha$, and for $(2)$, $x+2y = 2\gamma \leq 2\alpha\leq \beta$.

    Otherwise, take $x=2\gamma - \beta$ and $y=\beta - \gamma$. Then $x,y\geq 0$. Since $\gamma\leq \delta$, $(3)$ and $(4)$ are satisfied. $x+2y = \beta$ and $2x+y = 3\gamma-\beta$, so $(1)$ and $(2)$ are also satisfied.

    Therefore, in either case, there is a non-negative solution to the system of inequalities.

    It is clear that if $\alpha$, $\beta$ and $\gamma$ are rational, then the solution in each case is also rational.
\end{proof}

\begin{theorem}
    For $a,b\geq 3$, an $\ROS(h_1 h_2^a h_3^b)$ exists if and only if the following are satisfied:
    \begin{enumerate}[label=\textup{(\arabic*)}]
        \item $(a-1)h_2 + bh_3\geq h_1$,
        \item $(b-1)h_3 + ah_2\geq h_1$,
        \item $a(a-1)h_2^2\geq ah_1h_2 - bh_1h_3$,
        \item $b(b-1)h_3^2\geq bh_1h_3 - ah_1h_2$,
        \item $a(a-1)h_2^2 + b(b-1)h_3^2 \geq abh_2h_3 - 2ah_1h_2 + bh_1h_3$, and
        \item $a(a-1)h_2^2 + b(b-1)h_3^2 \geq abh_2h_3 - 2bh_1h_3 + ah_1h_2$.
    \end{enumerate}
\end{theorem}
\begin{proof}
    By \Cref{cor: ROS iff SROS} and \Cref{lemma: h_1h_2^ah_3^b inequalities}, we need only show that there are values for $X'(2,2,3)$ and $X'(2,3,3)$ satisfying the inequalities given in \Cref{lemma: h_1h_2^ah_3^b inequalities}. Observe that these inequalities are in the form given in \Cref{lemma: h_1h_2^ah_3^b inequality system solution} with $x=ab(a-1)X'(2,2,3)$ and $y=ab(b-1)X'(2,3,3)$, and so we must only satisfy the conditions given there.

    Thus, we must show that:
    \begin{enumerate}[label=\textup{\{\arabic*\}}]
        \item $a(a-1)h_2^2 - ah_1h_2 + abh_2h_3\geq 0$,
        \item $b(b-1)h_3^2 - bh_1h_3 + abh_2h_3\geq 0$,
        \item $abh_2h_3\geq 0$,
        \item $\max\{abh_2h_3-ah_1h_2,abh_2h_3-bh_1h_3\} \leq a(a-1)h_2^2 - ah_1h_2 + abh_2h_3$,
        \item $\max\{abh_2h_3-ah_1h_2,abh_2h_3-bh_1h_3\} \leq b(b-1)h_3^2 - bh_1h_3 + abh_2h_3$,
        \item $\max\{abh_2h_3-ah_1h_2,abh_2h_3-bh_1h_3\} \leq abh_2h_3$, and
        \item $a(a-1)h_2^2 + b(b-1)h_3^2 + 2abh_2h_3 - ah_1h_2 - bh_1h_3\geq 3\max\{abh_2h_3-ah_1h_2,abh_2h_3-bh_1h_3\}$.
    \end{enumerate}

    Conditions $\{1\}$ and $\{2\}$ are equivalent to $(1)$ and $(2)$, and $\{3\}$ and $\{6\}$ are clearly always satisfied. By splitting the cases, $\{7\}$ is equivalent to $(5)$ and $(6)$. Similarly, $\{4\}$ is equivalent to $(3)$ and $a(a-1)h_2^2\geq 0$, which is always true, and $\{5\}$ is equivalent to $(4)$ and $b(b-1)h_3^2\geq 0$.

    Therefore, the conditions $\{1\}$-$\{7\}$ are equivalent to $(1)$-$(6)$, and so an $\ROS(h_1 h_2^a h_3^b)$ exists if and only if $(1)$-$(6)$ are satisfied.
\end{proof}

It can be observed that the conditions $(1)$-$(6)$ are obtained from \Cref{condition: sufiĉa?}, and so the conditions of that result are sufficient to show the existence of an $\ROS(h_1 h_2^a h_3^b)$.

\section{Constructing realizations from rational outline squares}

    Although \Cref{thm: h_1h_2h_3^m rational} solves the existence of a symmetric rational outline square for partitions of the form $(h_1h_2h_3^m)$, it does not solve the existence of an $\LS(h_1h_2h_3^m)$. In this section, we construct an outline square for the realization from the rational outline square.

We construct a realization using the following lemma from \cite{kemp2024latin}, where $\{x\} = x - \lfloor x\rfloor$ denotes the fractional part of a real number $x$.

\begin{lemma}[\cite{kemp2024latin}]
\label{lemma: rational outline to full}
    Let $O$ be an $\ROS(h_1\dots h_k)$. If there exists a $k\times k$ array $B$ of multisets, where $y(i,j)$ denotes the number of entries in cell $(i,j)$ of $B$, such that
    \begin{itemize}
        \item $y(i,j) = \sum_{\ell\in[k]} \{O(i,j,\ell)\}$,
        \item the number of copies of symbol $j$ in row $i$ of $B$ is $y(i,j)$,
        \item and the number of copies of symbol $j$ in column $i$ of $B$ is $y(i,j)$,
    \end{itemize}
    then there exists an outline square $L$ which respects $(h_1\dots h_k)$.
\end{lemma}

We introduce some notation to better define the array $B$ that is needed for this construction method.

\begin{definition}
    A \emph{frequency array} $F$ of order $k$ is a $k\times k$ array, where each cell contains a single non-negative integer.
\end{definition}

\begin{definition}
    Let $O$ be a $k\times k$ array of multisets with elements from $[k]$. $O(i,j)$ denotes the multiset of symbols in cell $(i,j)$, $O_{\ell}^i$ and $^jO_\ell$ denote the number of copies of symbol $\ell$ in row $i$ and column $j$ respectively.
    Then $O$ is an \emph{outline array} corresponding to a frequency array $F$ of order $k$, if for all $i,j,\ell\in[k]$
    \begin{itemize}
        \item $|O(i,j)| = F(i,j)$,
        \item $O_\ell^i = F(i,\ell)$, and
        \item $^jO_\ell = F(\ell,j)$.
    \end{itemize}
\end{definition}

It is clear that the array $B$ in \Cref{lemma: rational outline to full} is an outline array for the frequency array $F$ where $F(i,j) = \sum_{\ell\in[k]} \{O(i,j,
\ell)\} = y(i,j) = y(j,i)$. Since the outline array $B$ is part of the final outline square $L$, $B$ can be considered a partial outline square. In general, any outline array is a partial outline square, and the conditions on outline arrays mean that they can be combined to form larger outline arrays. We now provide some results in that direction.

\subsection{Frequency arrays}

\begin{lemma}
\label{lemma: sum freq arrays}
    If $O_1$ and $O_2$ are outline arrays corresponding to the frequency arrays $F_1$ and $F_2$ respectively, then there exists an outline array $O^*$ corresponding to the frequency array $F^*$ where $F^*(i,j) = F_1(i,j) + F_2(i,j)$.
\end{lemma}
\begin{proof}
    Take $O^*$ to be the cell-wise union of $O_1$ and $O_2$. So $O^*(i,j) = O_1(i,j)\cup O_2(i,j)$ for all $i,j\in[k]$. Then $|O^*(i,j)| = |O_1(i,j)| + |O_2(i,j)| = F_1(i,j) + F_2(i,j) = F^*(i,j)$. Similarly, $(O^*)_\ell^i = F^*(i,\ell)$ and $^j(O^*)_\ell = F^*(\ell,j)$.
\end{proof}

Observe that a realization $\LS(h_1\dots h_k)$ is the sum of two outline arrays corresponding to the frequency arrays $F_1$ and $F_2$, where $F_1(i,j) = h_ih_j$ for all $i\neq j$, $F_1(i,i) = 0$, $F_2(i,j) = 0$ when $i\neq j$ and $F_2(i,i) = h_i^2$.

\begin{lemma}
\label{lemma: summing rows/columns of outline array}
    If an outline array $O$ exists for an order $k$ frequency array $F$, then for any partition $S_1,S_2,\dots,S_{k'}$ of $[k]$, an outline array $O^*$ exists for the order $k'$ array $F^*$, where for all $i,j\in[k']$ $$F^*(i,j) = \sum_{x\in S_i}\sum_{y\in S_j}F(x,y).$$
\end{lemma}
\begin{proof}
    The outline array $O^*$ is constructed by amalgamating the rows, columns and symbols of $O$ according to the partition $S$ of $[k]$. For all $i,j\in[k']$, take $$O^*(i,j) = \bigcup_{x\in S_i}\bigcup_{y\in S_j} O(i,j).$$
    Then it is clear that $|O^*(i,j)| = F^*(i,j)$.

    To obtain an array with only $k'$ symbols, for each $i\in[k']$ replace all symbols $x\in S_i$ with the symbol $i$. Thus, for all $i,\ell\in [k']$, ${O^*}_\ell^i = \sum_{r\in S_i}\sum_{x\in S_\ell}O_x^r = F^*(i,\ell)$. Similarly, $^j{O^*}_\ell = F^*(\ell,j)$ for all $j,\ell\in[k']$.
\end{proof}

\subsection{Equitable graph colouring}

Within the construction of the array $B$, we also use some graph colouring methods. Thus, we introduce those notions here. Let $G$ be a graph with vertex set $V$ and edge set $E$, where $E$ may contain repeated edges but no loops.

A partition of $E$ into mutually disjoint sets $C_1,\dots, C_k$ is an \emph{edge-colouring} of $G$ with $k$ colours, where edge $e\in E$ has colour $i$ if $e\in C_i$. For $v\in V$ and $i\in[k]$, let $c_i(v)$ be the number of edges adjacent to $v$ in $C_i$. An edge-colouring of $G$ is \emph{equitable} if for all $v\in V$ and all $i,j\in[k]$ where $i\neq j$,
$$|c_i(v)-c_j(v)|\leq 1.$$

The following theorem was proven by De Werra.

\begin{theorem}[\cite{de1971balanced}]
\label{thm: de werra}
    For each $k\geq 1$, any finite bipartite graph has an equitable edge-colouring with $k$ colours.
\end{theorem}

\subsection{Constructing the array $B$}

We return to the construction of the outline array $B$ required for \Cref{lemma: rational outline to full}.

\begin{observation}
\label{lemma: freq array for h_1h_2h_3^m}
Permute the rows and columns of the $\SROS(h_1h_2h_3^m,\{h_3^m\})$ in \Cref{thm: h_1h_2h_3^m rational}, so that it is now an $\SROS(h_3^mh_1h_2,\{h_3^m\})$. Then the frequency array $F$ for the outline array $B$ in \Cref{lemma: rational outline to full} is of order $k$ with
$$F(i,j) = \begin{cases}
    0, & \text{if $i=j$,}\\
    a, & \text{if $i,j\leq m$,}\\
    b, & \text{if $i=m+1$ and $j\leq m$ or $j=m+1$ and $i\leq m$,}\\
    c, & \text{if $i=m+2$ and $j\leq m$ or $j=m+2$ and $i\leq m$,}\\
    d, & \text{otherwise,}
\end{cases}$$
where $$a = (m-2)\{X'(3,3,3)\} + \{X'(1,3,3)\} + \{X'(2,3,3)\},$$ $$b = (m-1)\{X'(1,3,3)\} + \{X'(1,2,3)\},$$ $$c = (m-1)\{X'(2,3,3)\} + \{X'(1,2,3)\}$$ $$\text{and } d = m\{X'(1,2,3)\}.$$
\end{observation}

The frequency array $F$ is shown in \Cref{fig:abcd array}. Note that $0\leq a,b,c,d\leq m-1$, since $0\leq \{x\}<1$ for any $x\in\mathbb{Q}$.

    \begin{figure}[h]
        \centering
    $$\begin{array}{|cccccc|cc|} \hline
    - & a & a & \dots & a & a & b & c \\
    a & - & a & \dots & a & a & b & c \\
    a & a & - & \dots & a & a & b & c \\
    \vdots & \vdots & \vdots & \ddots & \vdots & \vdots & \vdots & \vdots \\
    a & a & a & \dots & - & a & b & c \\
    a & a & a & \dots & a & - & b & c \\ \hline
    b & b & b & \dots & b & b & - & d \\
    c & c & c & \dots & c & c & d & - \\ \hline
    \end{array}$$
        \caption{}
        \label{fig:abcd array}
    \end{figure}

\begin{lemma}
\label{L fill cond}
    Given positive integers $h_1,h_2,h_3$, the values of $a,b,c,d$ in \Cref{lemma: freq array for h_1h_2h_3^m} satisfy the following conditions:
    \begin{enumerate}[label=\textup{(\arabic*)}]
        \item $d\leq mb$
        \item $d\leq mc$
        \item $b\leq (m-1)a+c$
        \item $c\leq (m-1)a+b$
        \item $2d\geq m(b+c-(m-1)a)$
    \end{enumerate}
\end{lemma}
\begin{proof}
    Given that each $\{X'(i,j,\ell)\}$ is non-negative, each inequality follows immediately from \Cref{lemma: freq array for h_1h_2h_3^m}.
\end{proof}

    The five conditions in Lemma \ref{L fill cond} are assumed to be true throughout the rest of this section.

\begin{lemma}
\label{lemma: freq array with d in corner}
    For $m,d,k\in\mathbb{Z}$, if $m\geq 3$, $0\leq d\leq m$ and $m+2\leq k\leq 2m-1$, then there exists an outline array for the frequency array $F$ of order $k$, where
    $$F(i,j) = \begin{cases}
        1, & \text{if $i\leq m$ or $j\leq m$ and $i\neq j$,}\\
        d, & \text{if $(i,j)$ is $(m+1,m+2)$ or $(m+2,m+1)$,}\\
        0, & \text{otherwise.}
    \end{cases}$$
    If $m\geq 4$ and $d\geq \frac{m}{2}$, then there also exists an outline array for the same $F$, where $k=2m$.
\end{lemma}
\begin{proof}
    We first consider $k\leq 2m-1$. Let the outline array $O$ be split into subarrays as shown in \Cref{fig: split outline array}, where the columns (and rows) are of widths $m$, $2$ and $k-(m+2)$, in that order. The arrays $\emptyset_i$ are empty arrays of the appropriate sizes, and the arrays $A$ and $D$ are empty along the main diagonal. The remaining cells of $A$ contain one symbol each, and the remaining cells of $D$ contain $d$ symbols. The cells of the arrays $B_i$ and $C_i$ each contain a single symbol.
    \begin{figure}[h]
        \centering
        \begin{tikzpicture}
            \draw (0,0) rectangle (4,-4);
            \draw (1.5,0) -- (1.5,-4);
            \draw (2.5,0) -- (2.5,-4);
            \draw (0,-1.5) -- (4,-1.5);
            \draw (0,-2.5) -- (4,-2.5);

            \node at (0.75,-0.75) {$A$};
            \node at (2,-0.75) {$B_1$};
            \node at (0.75,-2) {$B_2$};
            \node at (3.25,-0.75) {$C_1$};
            \node at (0.75,-3.25) {$C_2$};
            \node at (2,-2) {$D$};
            \node at (3.25,-2) {$\emptyset_1$};
            \node at (2,-3.25) {$\emptyset_2$};
            \node at (3.25,-3.25) {$\emptyset_3$};
        \end{tikzpicture}
        \caption{Structure of the outline array $O$}
        \label{fig: split outline array}
    \end{figure}

    Since $F(i,i) = 0$ for all $i\in[k]$, we require that symbol $i$ does not appear row $i$ or column $i$. Further, to satisfy the frequency array, $C_1$, $C_2$ and $D$ must contain only symbols from $[m]$, $B_1$ and $B_2^T$ contain $m-d$ symbols from $[m]$ in each column and the rest of the entries are $m+1$ or $m+2$, and $A$ must contain $m$ copies of each symbol in $[k]\setminus[m+2]$ (with one copy in each row and column) as well as some entries from $[m+2]$.

    We start by placing $2m+1-k$ entries in each row and column of $A$, leaving the main diagonal empty and making sure that $i$ does not appear in row or column $i$, and we also fill all entries in column 1 of $B_1$ and $B_2^T$ with $m+2$ and all entries in column 2 with $m+1$. We then swap $m-d$ entries in each of these columns with entries in $A$.

    When $m\neq 6$, take $A$ to be a latin square of order $m$ with $m$ disjoint transversals. Permute $A$ so that the cells of one transversal form the main diagonal, and $A(i,i) = i$ for all $i\in[m]$.

    Remove the symbols in the main diagonal transversal and in another $k-(m+2)$ of the transversals. This means that the cells of the main diagonal are empty and symbol $i$ does not occur in row or column $i$. There are $m-1-k+(m+2) = 2m+1-k$ transversals left in $A$, and since $k\leq 2m-1$, it follows that $2m+1-k \geq 2$. Denote two of the remaining transversals by $T_1$ and $T_2$.

    If $m=6$, take $A$ to be a latin square with 4 transversals, such as the one given in \Cref{fig: order 6 square with 4 transversals}. As before, permute this so that one of the transversals appears along the main diagonal with symbols in increasing order, and then delete the entries in the main diagonal. Denote two of the remaining transversals by $T_1$ and $T_2$. We remove entries of $A$ so that there are $2m+1-k$ symbols left in each row and column, and $2m+1-k$ of each symbol, where $T_1$ and $T_2$ remain and $2m+1-k = 13-k\leq 5$.

\begin{figure}[h]
    \centering
    $$\arraycolsep=2pt\begin{array}{|cccccc|}\hline
        1_a & 4 & 5_d & 6_b & 2_c & 3 \\
        3_b & 2_a & 6 & 1_d & 4 & 5_c \\
        5 & 1_c & 3_a & 2 & 6_d & 4_b \\
        6_c & 5_b & 1 & 4_a & 3 & 2_d \\
        4_d & 6 & 2_b & 3_c & 5_a & 1 \\
        2 & 3_d & 4_c & 5 & 1_b & 6_a \\ \hline
    \end{array}$$
    \caption{A latin square of order 6 with 4 disjoint transversals}
    \label{fig: order 6 square with 4 transversals}
\end{figure}
    
    If $2m+1-k = 2$, then delete all entries that are not in $T_1$ or $T_2$. If $2m+1-k = 3$, remove all entries which are not in the three non-diagonal transversals. For $2m+1-k = 4$, delete the entries of the transversal which is not $T_1$ or $T_2$. When $2m+1-k = 5$, $A$ already satisfies the requirements.

    Therefore, whether or not $m=6$, we have an array $A$ with an empty main diagonal, $2m+1-k$ symbols in each row and column, $2m+1-k$ copies of each symbol in $[m]$, including two transversals, and symbol $i$ not appearing in row or column $i$.

    For $\ell\in[2]$, let $T_\ell^i$ denote the symbol in row $i$ of the transversal $T_\ell$, and let $^jT_\ell$ denote the symbol in column $j$. For all $i\in[m]$ and $j\in[2]$, let $j'$ be the remaining symbol in $[2]\setminus j$, and fill $B_1$ as
    $$B_1(i,j) = \begin{cases}
        m+j' & \text{if $T_j^i\leq d$,}\\
        T_j^i & \text{otherwise.}
    \end{cases}$$
    Similarly, for $j\in[m]$ and $i\in[2]$, let $i'$ be the remaining symbol in $[2]\setminus i$ and fill $B_2$ by
    $$B_2(i,j) = \begin{cases}
        m+i' & \text{if $^jT_i\leq d$,}\\
        ^jT_i & \text{otherwise.}
    \end{cases}$$
    
    Also, for the transversal $T_\ell$, replace each entry greater than $d$ with $m+\ell'$, where $\ell'$ is the symbol remaining in $[2]\setminus\ell$. Therefore, each row of $A$ and $B_1$ together still contains $2m+1-k$ symbols from $[m]$, as does each column of $A$ and $B_2$ together, with each symbol occurring $2m+1-k$ times across all rows (or columns) of the array pairs and $i$ never occurring in row or column $i$. Also, $m+1$ and $m+2$ appear once in each row and column of the respective array pairs, with $d$ copies of each in $B_1$ and $d$ copies of each in $B_2$.

    Now that $A$, $B_1$ and $B_2$ have entries placed as described earlier, we must fill $D$, $C_1$, $C_2$ and the remaining cells of $A$.

    Fill the two off-diagonal cells of $D$ each with the symbols of $[d]$. Thus, there are $d$ symbols in each, and the rows of $B_2$ and $D$ together contain each symbol from $[m]$ exactly once, and the columns of $B_1$ and $D$ contain each symbol of $[m]$ once.

    Therefore, the arrays $B_1$, $B_2$ and $D$ are filled as required.

    In the array $A$, there are $k-(m+2)$ empty off-diagonal cells in each row and column. Also, note that there are $k-(m+2)$ symbols in $[k]\setminus[m+2]$ which each need to appear once in the first $m$ rows and first $m$ columns of $O$, to satisfy the frequency array $F$.
    
    We fill the rest of $A$ by taking a colouring of the edges of the graph $G$, where $G$ is a bipartite graph on the sets $U$ and $V$ with $U=V=[m]$ and there is an edge between $u\in U$ and $v\in V$ if and only if cell $A(u,v)$ is empty. Thus, each vertex has degree $k-m-2$, and so we take an equitable edge-colouring of $G$ with $k-m-2$ colours, which exists by \Cref{thm: de werra}. Since the $k-m-2$ colours must occur equitably on the $k-m-2$ edges adjacent to each vertex, each vertex has one edge of each colour. For each edge $(u,v)$, if $(u,v)$ is of colour $s\in[k-m-2]$, then place symbol $m+2+s$ in cell $(u,v)$ of $A$. Thus, each off-diagonal cell of $A$ is filled, and each symbol of $[k]\setminus[m+2]$ occurs once in every row and column.

    Finally, to fill the array $C_1$, observe that there are $2m+1-k$ symbols from $[m]$ in each row of $A$ and $B_1$ combined, and so there are $k-m-2$ of these symbols which need to appear in each row of $C_1$. Create a bipartite graph $G$ on the sets $U=V=[m]$, where there is an edge between $u$ and $v$ if and only if symbol $v$ needs to appear in row $u$ of $C_1$. As before, each vertex has degree $k-m-2$, so an equitable colouring with $k-m-2$ colours gives each vertex a single edge of each colour. Construct $C_1$ by filling $C_1(u,j)$ with symbol $v$ if there is an edge between $u$ and $v$ with colour $j\in[k-m-2]$. Thus, each cell of $C_1$ is filled, and each symbol of $[m]$ appears once in each of the first $m$ rows of $O$, with the exception that $i$ does not appear in row $i$.

    Repeat the same process to fill $C_2$, using the first $m$ columns instead of rows.

    Therefore, $O$ is an outline array for $F$, where $k\leq 2m-1$.
    \vspace{0.125cm}

    We conclude the proof with a similar construction for the case $k = 2m$, where $m\geq 4$ and $d\geq \frac{m}{2}$. Let $O$ have the same structure shown in \Cref{fig: split outline array}. As before, we want the array $A$ to have $2m+1-k$ symbols of $[m+2]$ in each row and column. Here $2m+1-k = 1$, so we begin the construction of $A$ by placing a single transversal using symbols in $[m+2]$, while avoiding the main diagonal and making sure that symbol $i$ does not appear in row or column $i$.
    
    Take two sequences $r$ and $s$ of length $m$. For even $m$, define the sequences for all $i\in[m]$ by
    $$r_i = m+1-i,\quad\quad s_i = \begin{cases}
        i+1, & \text{for $i\notin\{\frac{m}{2},m\}$,}\\
        i+1-\frac{m}{2}, & \text{for $i\in\{\frac{m}{2},m\}$.}
    \end{cases}$$
    For odd $m$, let the sequences be
    $$r_i = \begin{cases}
        m+1-i, & \text{if $i\leq\frac{m-1}{2}$,}\\
        1, & \text{if $i=\frac{m+1}{2}$,}\\
        m+2-i, & \text{otherwise.}
    \end{cases} \quad\quad s_i = \begin{cases}
        i+1, & \text{for $i\notin\{\frac{m-1}{2},m\}$,}\\
        \frac{m+1}{2}, & \text{for $i=m$,}\\
        1, & \text{for $i=\frac{m-1}{2}$.}
    \end{cases}$$

    In the array $A$, for each $i\in[m]$, fill one cell by
    $$A(r_i,i) = \begin{cases}
        m+1, & \text{if $i\leq m-d$,}\\
        m+2, & \text{if $i > d$,}\\
        s_i, & \text{otherwise.}
    \end{cases}$$
    Then fill $B_1$ for $i\in[m]$ and $j\in[2]$ as
    $$B_1(r_i,j) = \begin{cases}
        m+2, & \text{if $j=1$ and $i\leq d$,}\\
        s_i, & \text{if $j=1$ and $i> d$,}\\
        m+1, & \text{if $j=2$ and $i> m-d$,}\\
        s_i, & \text{if $j=2$ and $i\leq m-d$.}
    \end{cases}$$
    Similarly, for $j\in[m]$ and $i\in[2]$, fill $B_2$ by
    $$B_2(i,j) = \begin{cases}
        m+2, & \text{if $i=1$ and $j\leq d$,}\\
        s_{m+1-j}, & \text{if $i=1$ and $j> d$,}\\
        m+1, & \text{if $i=2$ and $j> m-d$,}\\
        s_{m+1-j}, & \text{if $i=2$ and $j\leq m-d$.}
    \end{cases}$$

    Thus, in the array $D$, cell $(1,2)$ has symbols $s_i$ for all $m-d<i\leq m$, and cell $(2,1)$ has the symbols $s_i$ for all $1\leq i\leq d$.

    Now that $B_1$, $B_2$, and $D$ are filled, $A$ has $2m+1-k$ symbols in each row and column, and $O$ has $2m+1-k$ copies of each symbol in $[m]$, the rest of $O$ is filled using the bipartite graph method.

    Thus, there is an outline array $O$ for the frequency array $F$.
\end{proof}

\begin{lemma}
\label{lemma: fill L (a^mbc)}
    For the values $a$, $b$, $c$ and $d$ as given in \Cref{lemma: freq array for h_1h_2h_3^m}, the frequency array $F$ has a corresponding outline array $B$.
\end{lemma}
\begin{proof}
    If $a=0$ then by the conditions in \Cref{L fill cond} we are forced to have $b=c$. Using this, we then get that $2d\geq 2mb$, and so $m-1\geq d\geq mb$. This forces $b=0=c$, thus $d=0$ also. When all values of the frequency array $F$ are zero, we use the empty outline array $B$ and we are done. Thus, we assume from here that $a>0$.

    Since $2d\geq m(b+c - (m-1)a)$ and $d\leq m-1$, it is true that $(m-1)a \geq b+c - 2\frac{d}{m}\geq b+c - 2 + \frac{2}{m}$. 
    
    Let $A$ be a multiset of $b+c$ elements, where $m+1$ occurs $b$ times and $m+2$ occurs $c$ times. Let $A_1,\dots,A_a$ be a partition of $A$ into multisets.

    If $0<d<\frac{m}{2}$, then $(m-1)a\geq b+c-2\frac{d}{m}>b+c-1$. Thus, $(m-1)a\geq b+c$ and so can choose a partition of $A$ such that $|A_i|\leq m-1$ for all $i\in[a]$. If $\frac{m}{2}\leq d\leq m-1$, then $(m-1)a\geq b+c-2\frac{d}{m}>b+c-2$, and so $(m-1)a+1\geq b+c$. In this case, we  take $|A_i|\leq m-1$ for all $i\in[a]$ except $2\leq |A_1|\leq m$.

    If $d>0$, then $b,c>0$, so we further require that there is at least one copy each of $m+1$ and $m+2$ in $A_1$.

    Consider the case where $m=3$ and $\frac{m}{2}\leq d\leq m-1$. Then $d=2$. If $b+c \leq (m-1)a$ then $A$ can be partitioned with $|A_1|\leq m-1$. Otherwise, $b+c = 2a+1$, and since $b+c\leq 2(m-1) = 4$, it follows that $a=1$. In this case, there is only one multiset $A_1$ in the partition, and either $b=2$ and $c=1$ or $c=1$ and $b=2$.

    Let $A_n(m+1)$ be the number of copies of $m+1$ in $A_n$, and $A_n(m+2)$ be the same for $m+2$. For each $n\in[a]$, let $F_n$ be a frequency array of order $m+2$ where
    $$F_1(i,j) = \begin{cases}
        0 & \text{if $i=j$,}\\
        1 & \text{if $i,j\leq m$,}\\
        A_1(i) & \text{if $i>m$ and $j\leq m$,}\\
        A_1(j) & \text{if $j>m$ and $i\leq m$,}\\
        d & \text{if $(i,j)$ is $(m+1,m+2)$ or $(m+2,m+1)$,}\\
        0 & \text{otherwise.}
    \end{cases}$$
    and for $n>1$,
    $$F_n(i,j) = \begin{cases}
        0 & \text{if $i=j$,}\\
        1 & \text{if $i,j\leq m$,}\\
        A_n(i) & \text{if $i>m$ and $j\leq m$,}\\
        A_n(j) & \text{if $j>m$ and $i\leq m$,}\\
        0 & \text{otherwise.}
    \end{cases}$$

    Consider the case where $m=3$ and $|A_1| = m$. As shown earlier, this case forces $d=2$, $a=1$ and $\{A_1(m+1),A_1(m+2)\} = \{1,2\}$. In either case, the appropriate outline array $O_1$, corresponding to $F_1$, is given in \Cref{fig: $m=3$ case}.

    \begin{figure}[h]
        \centering
        $\begin{array}{|ccc|cc|}\hline
            - & 5 & 4 & 2,3 & 4 \\
            3 & - & 4 & 1,5 & 4 \\
            4 & 4 & - & 2,5 & 1 \\ \hline
            2,5 & 1,3 & 1,5 & - & 2,3 \\
            4 & 4 & 2 & 1,3 & - \\ \hline
        \end{array}$ \quad\quad\quad
        $\begin{array}{|ccc|cc|}\hline
            - & 4 & 5 & 5 & 2,3 \\
            3 & - & 5 & 5 & 1,4 \\
            5 & 5 & - & 1 & 2,4 \\ \hline
            5 & 5 & 2 & - & 1,3 \\
            2,4 & 1,3 & 1,4 & 2,3 & - \\ \hline
        \end{array}$
        \caption{The outline arrays $O_1$ when $m=3$ and $|A_1| = m$}
        \label{fig: $m=3$ case}
    \end{figure}
    
    In general, to construct the outline array $O_n$ corresponding to $F_n$, we first apply \Cref{lemma: freq array with d in corner} with order $k = m + |A_n|$ and parameter $d$ in \Cref{lemma: freq array with d in corner} equal to $d$ as defined here when $n=1$, and otherwise 0. We then apply \Cref{lemma: summing rows/columns of outline array} with any partition $S$ where $S_i = \{i\}$ for all $i\in[m]$, $|S_{m+1}| = A_n(m+1)$ and $|S_{m+2}| = A_n(m+2)$. When $n=1$ and $d>0$, we established above that $A_n(m+1),A_n(m+2)\geq 1$, and we require here that $m+1\in S_{m+1}$ and $m+2\in S_{m+2}$.

    Summing the outline arrays $O_n$ gives an outline array for the frequency array $F$, where $F = \sum_{n=1}^a F_n$. Therefore, by \Cref{lemma: sum freq arrays}, there is an outline array for $F$ also, and so the array $B$ exists.
\end{proof}

\begin{theorem}
    An $\LS(h_1h_2h_3^m)$ exists for $m\geq 3$ if and only if $h_1\leq mh_3$, $h_2\leq mh_3$, and $h_3^2 - \frac{h_3}{m-1}(h_1+h_2) + \frac{2}{m(m-1)}h_1h_2 \geq 0$.
\end{theorem}
\begin{proof}

    It is clear that the conditions are necessary from \Cref{lemma: always symmetric solution} and \Cref{thm: h_1h_2h_3^m rational}.

    Using \Cref{thm: h_1h_2h_3^m rational}, the necessary conditions are sufficient to construct a symmetric rational outline square $O$. If the array $B$ in \Cref{lemma: rational outline to full} exists then there exists an outline square for this realization.

    $B$ always exists by \Cref{lemma: fill L (a^mbc)}, and so it is always possible to bring the floored rational outline square $O$ to an integer solution.
\end{proof}

\section{Concluding remarks}

The necessary structure of a realization varies considerably depending on the number and size variation of the subsquares. This has led to constructions which allow for only heavily restricted forms of partition, and made the formulation of general necessary conditions a challenging problem. In this work we demonstrate the value of rational outline squares in  addressing both of these problems, providing significant new results on both the existence and necessary conditions for realizations.

The existence of a symmetric rational outline square is necessary for the existence of a realization, and it also provides a construction for a sufficiently large scalar multiple of the partition. It is clear that the existence of an $\LS(h_1\dots h_k)$ and an $\ROS(h_1\dots h_k)$ are related problems, and it is possible that the existence of a symmetric rational outline square is also sufficient.

Although outline squares have been used before to find realizations, the extra conditions added in an $\ROS(h_1\dots h_k)$ and $\SROS(h_1\dots h_k)$ can make it easier to find such outline squares. We have demonstrated this with the constructions targeting \Cref{conj: k-2,conj: largest 3} and partitions with many repeated parts. The use of similar symmetric rational outline squares also adds structure to the array $B$ that is leftover when the initial rational outline square is floored, and this allowed for the construction of all $\LS(h_1h_2^ah_3^b)$.

Certainly the most general result is the improved necessary condition in \Cref{condition: sufiĉa?}, which subsumes and improves upon all known conditions. This condition is sufficient for all $\LS(h_1\dots h_k)$ and $\ROS(h_1\dots h_k)$ discussed here, and it is of interest to determine if it is sufficient in general.

\printbibliography

\end{document}